\documentclass[11pt]{amsart}  

\usepackage[latin1]{inputenc}
\usepackage{amsmath}
\usepackage{amsfonts}
\usepackage{amssymb}
\usepackage{latexsym}
\usepackage{graphicx}
\usepackage{subfigure}
\usepackage{color}
\usepackage{mathtools}
\usepackage{verbatim}
\usepackage[all]{xy}
\usepackage{graphics}
\usepackage{pdfsync}
\usepackage{hyperref}
\usepackage{ytableau}

\theoremstyle{plain}
\newtheorem{theorem}{Theorem}[section]
\newtheorem{proposition}[theorem]{Proposition}

\newtheorem{lemma}[theorem]{Lemma}
\newtheorem{corollary}[theorem]{Corollary}
\newtheorem{conjecture}[theorem]{Conjecture}

\theoremstyle{definition}
\newtheorem{definition}[theorem]{Definition}
\newtheorem{example}[theorem]{Example}

\theoremstyle{remark}
\newtheorem{remark}[theorem]{Remark}

\numberwithin{equation}{section}

\newcommand{\bC}{\mathbb{C}}

\newcommand{\bQ}{\mathbb{Q}}
\newcommand{\bR}{\mathbb{R}}
\newcommand{\bZ}{\mathbb{Z}}

\newcommand{\Qsym}{\ensuremath{\operatorname{QSym}}}

\newcommand{\set}{\mathrm{Set}} 
\newcommand{\des}{\mathrm{Des}} 
\newcommand{\comp}{\mathrm{comp}} 

\newcommand{\suchthat}{\;|\;}

\newlength\cellsize 
\savebox2{%
\begin{picture}(15,15)
\put(0,0){\line(1,0){15}}
\put(0,0){\line(0,1){15}}
\put(15,0){\line(0,1){15}}
\put(0,15){\line(1,0){15}}
\end{picture}}
\newcommand\cellify[1]{\def\thearg{#1}\def\nothing{}%
\ifx\thearg\nothing
\vrule width0pt height\cellsize depth0pt\else
\hbox to 0pt{\usebox2\hss}\fi%
\vbox to 15\unitlength{
\vss
\hbox to 15\unitlength{\hss$#1$\hss}
\vss}}
\newcommand\tableau[1]{\vtop{\let\\=\cr
\setlength\baselineskip{-16000pt}
\setlength\lineskiplimit{16000pt}
\setlength\lineskip{0pt}
\halign{&\cellify{##}\cr#1\crcr}}}
\savebox3{%
\begin{picture}(15,15)
\put(0,0){\line(1,0){15}}
\put(0,0){\line(0,1){15}}
\put(15,0){\line(0,1){15}}
\put(0,15){\line(1,0){15}}
\end{picture}}
\newcommand\expath[1]{%
\hbox to 0pt{\usebox3\hss}%
\vbox to 15\unitlength{
\vss
\hbox to 15\unitlength{\hss$#1$\hss}
\vss}}
\newcommand\bas[1]{\omit \vbox to \cellsize{ \vss \hbox to \cellsize{\hss$#1$\hss} \vss}}

\usepackage{tikz}
\usetikzlibrary{arrows,petri,topaths}
\usepackage{tkz-berge}
\usepackage{xcolor}
\usepackage{tikz-qtree}
\usetikzlibrary{positioning,arrows,calc,intersections,shapes,decorations}

\newcommand{\sym}{\mathrm{Sym}} 

\newcommand{\anti}{\operatorname{\mathrm{Anti}}}
\newcommand{\schub}{\mathfrak{S}}
\newcommand{\cat}{\mathrm{Cat}}
\newcommand{\ds}[2]{\big\langle {#1}\big\rangle_{#2}}

\newcommand{\alpy}{\mathbf{y}} 
\newcommand{\alpx}{\mathbf{x}}
\newcommand{\alpxn}[1]{\mathbf{x}_{#1}}
\newcommand{\degn}{\mathcal{R}_n} 

\newcommand{\wc}[2]{\mathcal{W}_{#1}^{(#2)}}
\newcommand{\wcp}[1]{\mathcal{W}_{#1}^{'}}
\newcommand{\catwc}[1]{\mathcal{CW}_{#1}} 
\newcommand{\coset}{\mathrm{Cos}}
\newcommand{\code}[1]{\mathrm{code}(#1)} 

\newcommand{\pn}[1]{\textcolor{red}{#1}}

\begin{document}

\title[Divided symmetrization]{Divided symmetrization and quasisymmetric functions}
\author{Philippe Nadeau}
\address{Univ Lyon, Universit\'e Claude Bernard Lyon 1, CNRS UMR
5208, Institut Camille Jordan, 43 blvd. du 11 novembre 1918, F-69622 Villeurbanne cedex, France}
\email{\href{mailto:nadeau@math.univ-lyon1.fr}{nadeau@math.univ-lyon1.fr}}

\author{Vasu Tewari}
\address{Department of Mathematics, University of Pennsylvania, Philadelphia, PA 19104, USA}

\thanks{The second author was supported by an AMS-Simons travel grant.}
\email{\href{mailto:vvtewari@math.upenn.edu}{vvtewari@math.upenn.edu}}



\begin{abstract}
Motivated by a question in Schubert calculus, we study the interplay of quasisymmetric polynomials with the \emph{divided symmetrization} operator, which was introduced by Postnikov in the context of volume polynomials of permutahedra. Divided symmetrization is a linear form which acts on the space of polynomials in $n$ indeterminates of degree $n-1$. We first show that divided symmetrization applied to a quasisymmetric polynomial in $m$ indeterminates can be easily determined. Several examples with a strong combinatorial flavor are given. Then, we prove that the divided symmetrization of any polynomial can be naturally computed with respect to a direct sum decomposition due to Aval-Bergeron-Bergeron, involving the ideal generated by positive degree quasisymmetric polynomials in $n$ indeterminates.
\end{abstract}

\maketitle

\section{Introduction}
In his seminal work \cite{Pos09}, Postnikov  introduced an operator called divided symmetrization that plays a key role in computing volume polynomials of permutahedra.
This operator takes a polynomial $f(x_1,\dots,x_n)$ as input and outputs a symmetric polynomial $\ds{f(x_1,\dots,x_n)}{n}$ defined by
\[
  \ds{f(x_1,\dots,x_n)}{n}\coloneqq \sum_{w\in S_n}w\cdot\left(\frac{f(x_1,\dots,x_n)}{\prod_{1\leq i\leq n-1}(x_i-x_{i+1})}\right),
\]
where $S_n$ denotes the symmetric group on $n$ letters, naturally acting by permuting variables. When $f$ has degree $n-1$, its divided symmetrization $\ds{f}{n}$ is a scalar. Given ${\bf a}=(a_1,\dots,a_n)\in \bR^n$, the permutahedron $\mathcal{P}_{{\bf a}}$ is the convex hull of all points of the form $(a_{w(1)},\dots, a_{w(n)})$ where $w$ ranges over all permutations in $S_n$. Postnikov \cite[Section 3]{Pos09} shows that if $a_1\geq a_2\geq \cdots\geq a_n$, the volume of $\mathcal{P}_{{\bf a}}$ is given by $\frac{1}{(n-1)!}\ds{(a_1x_1+\cdots+a_nx_n)^{n-1}}{n}$. It is a polynomial in the $a_i$'s, and Postnikov goes on to give a combinatorial interpretation of its coefficients.

While a great deal of research has been conducted into various aspects of permutahedra, especially in regard to volumes and lattice point enumeration, divided symmetrization has received limited attention.
 Amdeberhan \cite{Amd16} considered numerous curious instances of divided symmetrization for various polynomials of all degrees.  Petrov \cite{Pet18} studied a more general divided symmetrization indexed by trees, which recovers Postnikov's divided symmetrization in the case the tree is a path.
 Amongst other results, Petrov provided a probabilistic interpretation involving sandpile-type model for certain remarkable numbers called mixed Eulerian numbers.\smallskip

Our own motivation for studying divided symmetrization stems from a problem in Schubert calculus, which we sketch now. The \emph{flag variety $Fl(n)$} is a complex projective variety structure on the set of complete flags, which are sequences $F_0=\{0\}\subset F_1\subset F_2\subset\cdots\subset F_{n-1}\subset F_n=\bC^n$ of subspaces such that $\dim F_i=i$. The Schubert varieties $X_w\subset Fl(n)$,  indexed by permutations in $S_n$, give rise to the basis of Schubert classes $\sigma_w$ in the integral cohomology $H^*(Fl(n))$.
The {\em Peterson variety $Pet_n$} is a subvariety of $Fl(n)$ of dimension $n-1$, appearing as a special case of a regular nilpotent Hessenberg variety. Our problem was to compute the number $a_w$ of points in its intersection with a generic translate of a Schubert variety $X_w$, for $w$ of length $n-1$. Equivalently, $a_w$ is  the coefficient of the class $[Pet_n]\in H^*(Fl(n))$ on the class $\sigma_w$.

 We show that $a_w$ is given by $\ds{\schub_w(x_1,\ldots,x_n)}{n}$ where $\schub_w$ is the celebrated Schubert polynomial attached to $w$. The results presented here are thus motivated by understanding the divided symmetrization of Schubert polynomials.\smallskip

In this article we set out to understand more about the structure of this operator acting on polynomials of degree $n-1$, since both Postnikov's work and our own have this condition.  Our investigations led us to uncover a direct connection between divided symmetrization and quasisymmetric polynomials. We now detail these results.

The ring of quasisymmetric functions in the infinite alphabet $\alpx=\{x_1,x_2,\dots\}$ was introduced by Gessel \cite{Ges84} and has since acquired great importance in algebraic combinatorics (all relevant definitions are recalled in Section~\ref{subsec:polynomials}). A distinguished linear basis for this ring is given by the fundamental quasisymmetric functions $F_{\alpha}$ where $\alpha$ is a composition.
Given a positive integer $n$, consider a quasisymmetric function $f(\alpx)$ of degree $n-1$.
We denote the quasisymmetric polynomial obtained by setting $x_i=0$ for all $i>m$ by $f(x_1,\dots,x_m)$ and refer to the evaluation of $f(x_1,\dots,x_m)$ at $x_1=\cdots=x_m=1$ by $f(1^m)$.
Our first main result states the following:
\begin{theorem}\label{thm:intro_main_1}
  For a quasisymmetric function $f$ of degree $n-1$, we have
  \begin{align*}
  \sum_{j\geq 1}f(1^j) t^j=\frac{\sum_{m=1}^{n}\ds{f(x_1,\ldots,x_m)}{n}t^m}{(1-t)^n}.
\end{align*}
\end{theorem}
Natural candidates for $f$ come from Stanley's theory of $P$-partitions \cite{St97,St99}: To any naturally labeled poset $P$ on $n-1$ elements, one can associate a quasisymmetric function $K_P(\alpx)$ with degree $n-1$.
Let $\mathcal{L}(P)$ denote the set of linear extensions of $P$.
Note that  elements in $\mathcal{L}(P)$ are permutations in $S_{n-1}$.
Under this setup, we obtain the following corollary of Theorem~\ref{thm:intro_main_1}.
\begin{corollary}
   For $m\leq n$, we have
  \[
    \ds{K_P(x_1,\dots,x_m)}{n}=|\{\pi \in \mathcal{L}(P)\suchthat \text{ $\pi$ has $m-1$ descents}\}|.
  \]
\end{corollary}

We further establish connections with an ideal of polynomials investigated by \cite{AB03,ABB04}. Let $\mathcal{J}_n$ denote the ideal in $\bQ[\alpx_n]\coloneqq\bQ[x_1,\dots,x_n]$  generated by homogeneous quasisymmetric polynomials in $x_1,\dots,x_n$ of positive degree.
Let $\degn$ be the degree $n-1$ homogeneous component of $\bQ[\alpx_n]$ , and let $K_n\coloneqq \degn\cap \mathcal{J}_n$.
Aval-Bergeron-Bergeron \cite{ABB04} provide a distinguished basis for a certain complementary space  $K_n^{\dagger}$ of $K_n$ in $\degn$, described explicitly in Section~\ref{subsec:structural2}.
This leads to our second main result.
\begin{theorem}\label{thm:intro_main_2}
If $f\in \degn$ is decomposed as $f=g+h$ with $g\in K_n^{\dagger}$ and $h\in K_n$ , then 
\[
\ds{f}{n}=g(1,\ldots,1).
\]
\end{theorem}
We conclude our introduction with  a brief outline of the article.
\medskip

\textbf{Outline of the article:} Section~\ref{sec:background} sets up the necessary notations and definitions.
In Section~\ref{sec:ds_basics} we gather several useful results concerning divided symmetrization that establish the groundwork for what follows.
Corollary~\ref{cor:fundamental_cor} and Lemma~\ref{lem:ds_monomial} are the key results of this section.
In Section~\ref{sec:ds_qs_poly} we focus on quasisymmetric polynomials, and Theorem~\ref{thm:intro_main_1} is proved in Section~\ref{subsec:all_the_action}.
In Section~\ref{subsec:applications} we apply our results to various fundamental-positive quasisymmetric functions that are ubiquitous in algebraic combinatorics.
In Section~\ref{sec:decomposition} deepens the connection with quasisymmetric polynomials by way of Theorem~\ref{thm:intro_main_2}, which gives a nice decomposition of divided symmetrization.

\section{Background}\label{sec:background}
We begin by recalling various standard combinatorial notions.
Throughout, for a nonnegative integer $n$, we set  $[n]\coloneqq \{i\suchthat1\leq i\leq n\}$. In particular, $[0]=\emptyset$.
We refer the reader to \cite{St97,St99} for any undefined terminology.

\subsection{Compositions}\label{subsec:compositions}
Given a nonnegative integer $k$ and a positive integer $n$, a \emph{weak composition} of $k$ with $n$ \emph{parts} is a sequence
$(c_1,\ldots,c_n)$ of nonnegative integers whose sum is $k$.
We denote the set of weak compositions of $k$ with $n$ parts by $\wc{n}{k}$. For the special case $k=n-1$ which will play a special role, we define $\wcp{n}=\wc{n}{n-1}$. 
Clearly, $|\wc{n}{k}|=\binom{n+k-1}{k}$. The \emph{size} of a weak composition ${\bf c}=(c_1,\dots,c_n)$ is the sum of its parts and is denoted by $|{\bf c}|$.
A \emph{strong composition} is a weak composition all of whose parts are positive.
Given a weak composition ${\bf c}$, we denote the underlying strong composition obtained by omitting zero parts by ${\bf c}^{+}$.
Henceforth, by the term composition, we always mean strong composition. Furthermore, we use boldface Roman alphabet for weak compositions and the Greek alphabet for compositions.
If the size of a composition $\alpha$ is $k$, we denote this by $\alpha\vDash k$.
We denote the number of parts of $\alpha$ by $\ell(\alpha)$.

Given $\alpha=(\alpha_1,\dots,\alpha_{\ell(\alpha)})\vDash k$ for $k$ a positive integer, we associate a subset $\set(\alpha)=\{\alpha_1,\alpha_1+\alpha_2,\dots,\alpha_1+\cdots+\alpha_{\ell(\alpha)-1}\}\subseteq [k-1]$.
Clearly, this correspondence is a bijection between compositions of $k$ and subsets of $[k-1]$.
Given $S\subseteq [k-1]$, we define $\comp(S)$ to be the composition of $k$ associated to $S$ under the preceding correspondence.
The inclusion order on subsets allows us to define the \emph{refinement order} on compositions. More specifically, given $\alpha$ and $\beta$ both compositions of $k$, we say that $\beta$ \emph{refines} $\alpha$, denoted by $\alpha\preccurlyeq \beta$, if $\set(\alpha)\subseteq \set(\beta)$.
For instance, we have $\alpha=(1,3,2,2)\preccurlyeq (1,2,1,1,1,2)=\beta$ as $\set(\alpha)=\{1,4,6\}$ is a subset of $\set(\beta)=\{1,3,4,5,6\}$.

\subsection{Polynomials}\label{subsec:polynomials}
Given a positive integer $n$, define two operators $\sym_n$ and $\anti_n$ that respectively symmetrize and antisymmetrize functions of the variables $x_1,\ldots,x_n$:
\begin{align*}
\sym_n(f(x_1,\ldots,x_n))&=\sum_{w\in S_n} f(x_{w(1)},\ldots,x_{w(n)}),\\
\quad \anti_n(f(x_1,\ldots,x_n))&=\sum_{w\in S_n}\epsilon(w) f(x_{w(1)},\ldots,x_{w(n)}).
\end{align*}
Here $\epsilon(w)$ denotes the \emph{sign} of the permutation $w$.
We denote the set of variables $\{x_1,\dots,x_n\}$ by $\alpx_n$.
Furthermore, set $\bQ[\alpx_n]\coloneqq\bQ[x_1,\dots,x_n]$.
Given a nonnegative integer $k$, let $\bQ^{(k)}[\alpx_n]$ denote the degree $k$ homogeneous component of $\bQ[\alpx_n]$.
Given a weak composition ${\bf c}=(c_1,\dots,c_{n})$, let
\[
\alpx^{{\bf c}}\coloneqq \prod_{1\leq i\leq n} x_i^{c_i}.
\]
Via the correspondence ${\bf c}\mapsto \alpx^{{\bf c}}$ for ${\bf c}=(c_1,\dots,c_n)\in \wc{n}{k}$, we see that $\wc{n}{k}$ naturally indexes a basis of the vector space $\mathbb{Q}^{(k)}[\alpx_n]$. In particular $\wcp{n}$ indexes the monomial basis of $\degn$. Recall from the introduction that we refer to $\mathbb{Q}^{(n-1)}[\alpx_n]$ as $\degn$.
 
Let $\Delta_n=\Delta(x_1,\ldots,x_n)=\prod_{1\leq i<j\leq n}(x_i-x_j)$ denote the usual Vandermonde determinant.
Given $f\in \bQ[\alpx_n]$, we say that $f$ is \emph{antisymmetric} if $w(f)=\epsilon(w)f$ for all $w\in S_n$.
Recall that if $f$ is antisymmetric, then it is divisible  by $\Delta_n$.
We say that $f$ is \emph{symmetric} if $w(f)=f$ for all $w\in S_n$.
The space of symmetric polynomials in $\bQ[\alpx_n]$ is denoted by $\Lambda_n$, and we denote its degree $d$ homogeneous component by $\Lambda_{n}^{(d)}$.
For the sake of brevity, we refer the reader to \cite[Chapter 7]{St99} and \cite{Mac95} for encyclopaedic exposition  on symmetric polynomials, in particular on the relevance of various bases of $\Lambda_n$ to diverse areas in mathematics.
Instead, we proceed to discuss the space of quasisymmetric polynomials, which includes $\Lambda_n$ and has come to occupy a central role in algebraic combinatorics since its introduction by Gessel \cite{Ges84}.

A polynomial $f\in \bQ[\alpx_n]$ is called \emph{quasisymmetric} if the coefficients of $\alpx^{{\bf a}}$ and $\alpx^{{\bf b}}$ in $f$ are equal whenever ${\bf a}^{+}={\bf b}^{+}$.
We denote the space of quasisymmetric polynomials in $x_1,\dots,x_n$ by $\Qsym_n$ and its degree $d$ homogeneous component by $\Qsym_n^{(d)}$.
A  basis for $\Qsym_n^{(d)}$ is given by the \emph{monomial quasisymmetric polynomials} $M_{\alpha}(x_1,\dots,x_n)$ indexed by compositions $\alpha\vDash d$.
More precisely, we set
\begin{equation}
\label{eq:defi_quasi_monomials}
  M_{\alpha}(x_1,\dots,x_n)=\sum_{\substack{{\bf a}\in \wc{n}{d}\\ {\bf a}^{+}=\alpha}}\alpx^{{\bf a}}.
\end{equation}
The reader may verify that $f=x_1^2x_2+x_1^2x_3+x_2^2x_3+x_1x_2x_3$
is a quasisymmetric polynomial in $\bQ[\alpxn{3}]$, and it can be expressed as $M_{(2,1)}(\alpx_3)+M_{1,1,1}(\alpx_3)$.
We note here that $M_{\alpha}(\alpx_n)=0$ if $\ell(\alpha)>n$.

Arguably the more important basis for $\Qsym_n$ consists of the \emph{fundamental quasisymmetric polynomials} $F_{\alpha}(x_1,\dots,x_n)$ indexed by compositions $\alpha$.
We set
\begin{equation}
\label{eq:Falpha}
F_{\alpha}(x_1,\dots,x_n)=\sum_{\alpha \preccurlyeq \beta}M_{\beta}(x_1,\dots,x_n).
\end{equation}
For instance $F_{(1,2)}(\alpx_3)=M_{(1,2)}(\alpx_3)+M_{(1,1,1)}(\alpx_3)$.

\section{Divided symmetrization}\label{sec:ds_basics}
We begin by establishing some basic results on divided symmetrization. 

\subsection{Basic properties}\label{subsec:basic_properties}
\begin{lemma}[\cite{Pos09}]
\label{lem:cancellation}
Let $f\in \bQ^{(k)}[\alpx_n]$ be a homogeneous polynomial.
  \begin{enumerate}
    \item If $k<n-1$, then $\ds{f}{n}=0$.
    \item If $k\geq n-1$, then $\ds{f}{n}\in \bQ^{(k-n+1)}[\alpx_n]$ is a symmetric polynomial.
  \end{enumerate}
\end{lemma}

\begin{proof}
 We will prove the following more general result:  Let $S=\{(i,j)\suchthat 1\leq i<j\leq n\}$, and for $I\subseteq S$  define $\Delta_I=\prod_{(i,j)\in I}(x_i-x_j)$. Then we claim that $\sym_n(f/\Delta_I)$ is $0$ is $k<|I|$ and is in $\bQ^{(k-|I|)}[\alpx_n]$ otherwise  (the lemma is the special case $I=\{(i,i+1)\suchthat i<n\}$).
 
  First, factor the Vandermonde $\Delta_n=\Delta_I\Delta_{I^c}$ where $I^c\coloneqq S\setminus I$.  Then 
\[\sym_n(f/\Delta_I)=\sym_n(f\Delta_{I^c}/\Delta_n)=\anti_n(f\Delta_{I^c})/\Delta_n.\]
The second identity follows from the fact that a permutation $w$ acts on $\Delta_n$ by the scalar $\epsilon(w)$. Now $\anti_n(f\Delta_{I^c})$ is an antisymmetric polynomial, and is thus divisible by $\Delta_n$, therefore $\sym_n(f/\Delta_I)$ is a symmetric polynomial. Since it has degree $k-|I|$ as a rational function, the proof follows.

\end{proof}
\noindent Since symmetric polynomials in $\Lambda_n$ act as scalars for $\sym_n$, we obtain the following corollary.
\begin{corollary}\label{cor:symmetric_factor}
 If $f$ has a homogeneous \textit{symmetric} factor of degree $>\deg(f)+1-n$, then $\ds{f}{n}=0$.
\end{corollary}
\noindent Divided symmetrization behaves nicely with respect to reversing or negating the alphabet, as well as adding a constant to each letter in our alphabet, as the next lemma states. We omit the straightforward proof.
\begin{lemma}\label{lem:reversal_and_translation_invariance}
  Let $f(x_1,\dots,x_n)\in \bQ[\alpx_n]$ be homogeneous, and let $g(x_1,\dots,x_n)\coloneqq \ds{f(x_1,\dots,x_n)}{n}$. Suppose $c$ is an constant.
  We have the following equalities.
  \begin{enumerate}
  \item  $\ds{f(x_n,\dots,x_1)}{n}= (-1)^{n-1}g(x_1,\dots,x_n)$.
  \item $\ds{f(-x_1,\dots,-x_n)}{n}=(-1)^{\deg(f)-n+1}g(x_1,\dots,x_n)$.
  \item $\ds{f(x_1+c,\dots,x_n+c)}{n}=g(x_1+c,\dots,x_n+c)$.
  \end{enumerate}
\end{lemma}

Our next lemma, whilst simple, is another useful computational aid. 
For a positive integer $i$ satisfying $1\leq i\leq n-1$, let $\coset(i,n-i)$ denote the set of permutations (in one-line notation) such that  $\sigma_1<\ldots<\sigma_i$ and $\sigma_{i+1}<\ldots<\sigma_n$. Equivalently, $\sigma$ is either the identity or has a unique descent in position $i$. $\coset(i,n-i)$ is known to be the set of minimal length representatives of the set of left cosets $S_n/S_i\times S_{n-i}$. For instance, if $n=4$ and $i=2$, then $\coset(2,2)$ equals $\{1234, 1324,1423,2314,2413,3412\}$.

\begin{lemma}
\label{lem:fundamental_lemma_general}
Let $f=(x_i-x_{i+1})g(x_1,\dots,x_i)h(x_{i+1},\ldots,x_{n})$ where $g,h$ are homogeneous.
Suppose that $\ds{ g(x_1,\dots,x_i)}{i}=p(x_1,\dots,x_i)$ and $\ds{h(x_1,\dots,x_{n-i})}{n-i}=q(x_1,\dots,x_{n-i})$.
Then
\[
\ds{f}{n} =\sum_{\sigma\in\coset(i,n-i)} p(x_{\sigma(1)},\dots, x_{\sigma(i)})q(x_{\sigma(i+1)},\dots, x_{\sigma(n)}).
\]
\end{lemma}
\begin{proof}
Under the given hypothesis, we have
\[
\frac{f(x_1,\dots,x_n)}{\prod_{1\leq j\leq n-1}(x_j-x_{j+1})}=\frac{g(x_1,\dots,x_i)}{\prod_{1\leq k\leq i-1}(x_k-x_{k+1})}\frac{h(x_{i+1},\dots,x_{n})}{\prod_{i+1\leq l\leq n-1}(x_l-x_{l+1})}
\]
By considering representatives of left cosets $S_n/S_i\times S_{n-i}$ we obtain
\begin{align*}
\ds{f}{n}&=\sum_{\sigma \in S_n/S_i\times S_{n-i}}\sigma\left( \sum_{\tau\in S_i\times S_{n-i}}\tau\left(\frac{g(x_1,\dots,x_i)}{\displaystyle\prod_{ k=1}^{ i-1}(x_k-x_{k+1})}\frac{h(x_{i+1},\dots,x_{n})}{\displaystyle\prod_{l=i+1}^{ n-1}(x_l-x_{l+1})}\right)\right)\nonumber\\
&=\sum_{\sigma \in S_n/S_i\times S_{n-i}}\sigma (p(x_1,\dots,x_i)q(x_{i+1},\dots,x_n)).
\end{align*}
The claim now follows.
\end{proof}

\noindent Lemma~\ref{lem:fundamental_lemma_general} simplifies considerably if $\deg(f)=n-1$, and we employ the resulting statement repeatedly throughout this article.
\begin{corollary}[\cite{Pet18}]\label{cor:fundamental_cor}
  Let $f\in \degn$ be such that
  \[f=(x_i-x_{i+1})g(x_1,\dots,x_i)h(x_{i+1},\ldots,x_n).\]
  Then \[\ds{f}{n}=\binom{n}{i} \ds{g(x_1,\dots,x_i)}{i}\ds{h(x_1,\ldots,x_{n-i})}{n-i}.\]
  In particular, $\ds{f}{n}=0$ if $\deg(g)\neq i-1$ (or equivalently $\deg(h)\neq n-i-1$).
\end{corollary}
\begin{proof}
  We first deal with the case $\deg(g)\neq i-1$. If $\deg(g)<i-1$, then Lemma~\ref{lem:cancellation} implies that $\ds{ g(x_1,\dots,x_i)}{i}=0$. If $\deg(g)>i-1$, then $\deg(h)< n-i-1$ and thus $\ds{h(x_1,\dots,x_{n-i})}{n-i}=q(x_1,\dots,x_{n-i})=0$, again by Lemma~\ref{lem:cancellation}.
  It follows by Lemma~\ref{lem:fundamental_lemma_general} that $\ds{f}{n}=0$ if $\deg(g)\neq i-1$.

  Now assume that $\deg(g)=i-1$.
  Then the polynomials $p(x_1,\dots,x_i)$ and $q(x_1,\dots,x_{n-i})$ in the statement of Lemma~\ref{lem:fundamental_lemma_general} are both constant polynomials by Lemma~\ref{lem:cancellation}, thereby implying
  \begin{align}
    \ds{f}{n}& =\sum_{\sigma\in\coset(i,n-i)} p(x_{\sigma(1)},\dots, x_{\sigma(i)})q(x_{\sigma(i+1)},\dots, x_{\sigma(n)})
    \nonumber\\
    &= \binom{n}{i} \ds{g(x_1,\dots,x_i)}{i}\ds{h(x_1,\ldots,x_{n-i})}{n-i},
  \end{align}
  where in arriving at the last equality we use the fact that  $|\coset(i,n-i)|=\binom{n}{i}$.
  \end{proof}

\begin{example}
\label{ex:basic_computation}
For $1\leq i\leq n$, let us show that 
\begin{equation}
\label{eq:evalXi} 
 \text{If  } X_i\coloneqq \prod_{\substack{1\leq j\leq n\\ j\neq i}}x_j,\text{  then  }\ds{X_i}{n}=(-1)^{n-i}\binom{n-1}{i-1}.
\end{equation}
We use induction on $n$. The claim is clearly true when $n=1$, where the empty product is to be interpreted as $1$.
Assume $n\geq 2$ henceforth.
By Corollary~\ref{cor:fundamental_cor}, for $i=1,\ldots,n-1$, we have
\begin{align*}
\ds{X_{i+1}-X_i}{n}&=\ds{(x_1\ldots x_{i-1})(x_i-x_{i+1})(x_{i+2}\cdots x_{n})}{n}\\
&=\binom{n}{i}\ds{x_1\ldots x_{i-1}}{i}\ds{x_{2}\ldots x_{n-i}}{n-i}.
\end{align*}
By the inductive hypothesis, we have $\ds{x_1\ldots x_{i-1}}{i}=1$ and $\ds{x_{2}\ldots x_{n-i}}{n-i}=(-1)^{n-i-1}$.
Therefore
\begin{align}\label{eqn:n-1_equations}
  \ds{X_{i+1}}{n}-\ds{X_i}{n}=(-1)^{n-i-1}\binom{n}{i}.
\end{align}
By summation, this gives~\eqref{eq:evalXi} up to a common additive constant. 

The proof is then complete using $\sum_{i=1}^{n}\ds{X_i}{n}=\ds{\sum_{i=1}^{n}X_i}{n}=0$, which follows from Corollary \ref{cor:symmetric_factor} because $\sum_{i=1}^{n}X_i$ is symmetric in $x_1,\dots,x_n$.
\end{example}

\subsection{Monomials of degree $n-1$}\label{subsec:ds_monomials}
 If $f=\alpx^{\bf c}$ where ${\bf c}\in \wcp{n}$, then \cite[Proposition 3.5]{Pos09} gives us a precise combinatorial description for $\ds{f}{n}$.
We reformulate this description, following Petrov~\cite{Pet18}: given ${\bf c}\coloneqq (c_1,\dots,c_n)\in \wcp{n}$, define the subset $S_{\bf c}\subseteq{[n-1]}$ by
\begin{equation}
\label{eq:Sc}
S_{\bf c}\coloneqq \{k\in\{1,\ldots,n-1\}\suchthat\sum_{i=1}^k c_i<k\}.
\end{equation}

Let $\mathrm{psum}_k({\bf c})\coloneqq\sum_{1\leq i\leq k}(c_i-1)$ for $i=0,\ldots,n$, so that by definition $k\in [n-1]$ belongs to $S_{\bf c}$ if and only if $\mathrm{psum}_k({\bf c})<0$. 

A graphical interpretation is helpful here: transform $c$ into a path $P({\bf c})$ from $(0,0)$ to $(n,-1)$ by associating a step $(1,c_i-1)$ to each $c_i$. For instance, Figure~\ref{fig:path_to_compute_Sc} depicts $P({\bf c})$ for  ${\bf c}=(0,3,0,0,0,1,3,0,)$. The successive $y$-coordinates of the integer points of $P({\bf c})$ are the values $\mathrm{psum}_k({\bf c})$, so that $S_{\bf c}$ consists of the abscissas of the points with negative $y$-coordinate. In our example, $S_{\bf c}=\{1,4,5,6\}\subseteq [7]$.

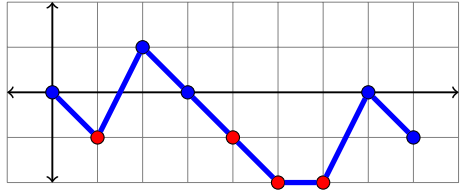
\begin{figure}[ht]
  \begin{tikzpicture}[scale=.6]
    \draw[gray,very thin] (0,0) grid (10,4);
    \draw[line width=0.25mm, black, <->] (0,2)--(10,2);
    \draw[line width=0.25mm, black, <->] (1,0)--(1,4);
    \node[draw, circle,minimum size=5pt,inner sep=0pt, outer sep=0pt, fill=blue] at (1, 2)   (b) {};
    \node[draw, circle,minimum size=5pt,inner sep=0pt, outer sep=0pt, fill=red] at (2, 1)   (c) {};
    \node[draw, circle,minimum size=5pt,inner sep=0pt, outer sep=0pt, fill=blue] at (3, 3)   (d) {};
    \node[draw, circle,minimum size=5pt,inner sep=0pt, outer sep=0pt, fill=blue] at (4, 2)   (e) {};
    \node[draw, circle,minimum size=5pt,inner sep=0pt, outer sep=0pt, fill=red] at (5, 1)   (f) {};
    \node[draw, circle,minimum size=5pt,inner sep=0pt, outer sep=0pt, fill=red] at (6, 0)   (g) {};
    \node[draw, circle,minimum size=5pt,inner sep=0pt, outer sep=0pt, fill=red] at (7, 0)   (h) {};
    \node[draw, circle,minimum size=5pt,inner sep=0pt, outer sep=0pt, fill=blue] at (8, 2)   (i) {};
    \node[draw, circle,minimum size=5pt,inner sep=0pt, outer sep=0pt, fill=blue] at (9, 1)   (j) {};
    \draw[blue, line width=0.7mm] (b)--(c);
    \draw[blue, line width=0.7mm] (c)--(d);
    \draw[blue, line width=0.7mm] (d)--(e);
    \draw[blue, line width=0.7mm] (e)--(f);
    \draw[blue, line width=0.7mm] (f)--(g);
    \draw[blue, line width=0.7mm] (g)--(h);
    \draw[blue, line width=0.7mm] (h)--(i);
    \draw[blue, line width=0.7mm] (i)--(j);
  \end{tikzpicture}
  \caption{$P({\bf c})$ when ${\bf c}=(0,3,0,0,0,1,3,0)$ with $S_{\bf c}=\{1,4,5,6\}$.}
  \label{fig:path_to_compute_Sc}
\end{figure}

For a subset $S\subseteq [n-1]$, let
\begin{align}
  \beta_n(S)\coloneqq |\{w\in S_n\suchthat \des(w)=S\}|,
\end{align}
where $\des(w)\coloneqq \{1\leq i\leq n-1\suchthat w_i>w_{i+1}\}$ is the set of descents of $w$.
Whenever the $n$ is understood from context, we simply say $\beta(S)$ instead of $\beta_n(S)$.
Postnikov \cite{Pos09} shows that $\ds{\alpx^{{\bf c}}}{n}$ for ${\bf c}\in \wcp{n}$ equals $\beta(S_{\bf c})$ up to sign.
His proof proceeds by computing constant terms in the Laurent series expansion of the rational functions occurring in the definition of $\ds{\alpx^{{\bf c}}}{n}$.
Petrov~\cite{Pet18} gives a more pleasing proof in a slightly more general context. 

\begin{lemma}[Postnikov]
\label{lem:ds_monomial}
If ${\bf c}=(c_1,\dots,c_n)\in \wcp{n}$, then
\begin{equation}
\label{eq:monomial_evaluation}
\ds{\alpx^{{\bf c}}}{n}=(-1)^{|S_{\bf c}|}\beta(S_{\bf c}).
\end{equation}
\end{lemma}

We give a proof based on Petrov's version in the appendix, which serves to illustrate the utility of Corollary~\ref{cor:fundamental_cor}.

\subsection{Catalan compositions and monomials}
\label{subsec:catalan}

We now focus on $\ds{\alpx^{\bf c}}{n}$ where ${\bf c}$ belongs to a special subset of $\wcp{n}$.
Consider $\catwc{n}$ defined as
\begin{align}
  \label{eqn:def_of_catalan_composition}
  \catwc{n}=\{{\bf c}\in \wcp{n}\suchthat \mathrm{psum}_k({\bf c})\geq 0 \text{ for } 1\leq k\leq n-1\}.
\end{align}
Equivalently, ${\bf c}=(c_1,\dots,c_n)\in \catwc{n}$ if $\sum_{i=1}^k c_i\geq k$ for all $1\leq k\leq n-1$ and $\sum_{i=1}^{n} c_i= n-1$. The paths $P({\bf c})$ for  ${\bf c}\in\catwc{n}$ are those that remain weakly above the $x$-axis except that the ending point has height $-1$. These are known as the (extended) {\em Lukasiewicz paths}. This description immediately implies that $|\catwc{n}|=\cat_{n-1}$, the $(n-1)$-th Catalan number equal to $\frac{1}{n}\binom{2n-2}{n-1}$.
In view of this, we refer to elements of $\catwc{n}$ as \emph{Catalan compositions}.
Observe that $S_{\bf c}=\emptyset$ if and only if ${\bf c}\in \catwc{n}$.
By Lemma~\ref{lem:fundamental_lemma_general}, we have that $\ds{\alpx^{\bf c}}{n}=1$ when ${\bf c}\in \catwc{n}$, since the only permutation whose descent set is empty is the identity permutation.
\begin{definition}
\label{def:catalan_monomials}
The monomials $\alpx^{\bf c}$ where ${\bf c}\in \catwc{n} $ are called \emph{Catalan monomials}.
\end{definition}
 For example, the Catalan monomials of degree $3$ obtained from elements of $\catwc{4}$ are given by $\{x_1^3,x_1^2x_2,x_1^2x_3,x_1x_2^2,x_1x_2x_3\}$.

\begin{remark}
  We refer to the image of a Catalan monomial under the involution $x_i\mapsto x_{n+1-i}$ for all $1\leq i\leq n$ as an \emph{anti-Catalan monomial}. These monomials are characterized by $S_{\bf c}=[n-1]$. By Lemma~\ref{lem:fundamental_lemma_general}, the divided symmetrization of an anti-Catalan monomial yields $(-1)^{n-1}$.
\end{remark}

The preceding discussion implies the following fact: If $f(x_1,\dots,x_n)\in \degn$ is such that each monomial appearing in $f$ is a Catalan monomial, then $\ds{f}{n}=f(1^n)$.
Here, $f(1^n)$ refers to the usual evaluation of $f(x_1,\dots,x_n)$ at $x_1=\cdots=x_n=1$.
This statement is a very special case of a more general result that we establish in Section~\ref{sec:decomposition}.

\section{Divided symmetrization of quasisymmetric polynomials}\label{sec:ds_qs_poly}
In this section we give a natural interpretation to the divided symmetrization of quasisymmetric polynomials of degree $n-1$.
\subsection{Quasisymmetric monomials}\label{subsec:quasi_monomial}
We focus first on monomial quasisymmetric polynomials. Recall that for any positive integer $m$, $\alpx_m$ refers the alphabet $\{x_1,\dots,x_m\}$. Our first result shows in particular that $\ds{M_{\alpha}(\alpx_m)}{n}$ depends solely on $n$, $m$ and $\ell(\alpha)$.
\begin{proposition}
\label{prop:DS_Malpha_general}
Fix a positive integer $n$, and let $\alpha\vDash n-1$. Then
\begin{align}\label{eqn:divided symmetrization monomial quasisymmetric}
  \ds{{M}_{\alpha}(\alpx_m)}{n}=(-1)^{m-\ell(\alpha)}\binom{n-1-\ell(\alpha)}{m-\ell(\alpha)}
\end{align}
for any $m\in\{\ell(\alpha),\ldots,n-1\}$, and
\begin{align}\label{eqn:divided symmetrization monomial quasisymmetric 0}
  \ds{{M}_{\alpha}(\alpx_n)}{n}=0.
\end{align}
\end{proposition}

Note that we could have a unified statement by defining $\binom{a}{b}$ to be $0$ if $b>a$, as usual. However it is useful to state the case $m=n$ separately since it plays a special role in the proof, and will be generalized in Lemma~\ref{lem:DS_quasisymmetric}.

\begin{proof}
  Our strategy is similar to that adopted in the proof of Lemma~\ref{lem:ds_monomial} given in Appendix \ref{app:proof}. That is, we perform elementary transformations on the compositions $\alpha$ of a given length, so that the value of $\ds{{M}_{\alpha}(\alpx_m)}{n}$ is preserved. Our transformations allow us to reach a `hook composition' for which we can compute the quantity of interest.

  We proceed by induction on $\ell\coloneqq\ell(\alpha)$.
  If $\ell=1$, then we have to show that
  \begin{equation}\label{eqn:first}
    \ds{M_{(n-1)}(\alpx_m)}{n}=\ds{x_1^{n-1}+\cdots +x_m^{n-1}}{n}=(-1)^{m-1}\binom{n-2}{m-1}.
  \end{equation}
  By Lemma~\ref{lem:ds_monomial}, we know that $ \ds{x_i^{n-1}}{n}=(-1)^{i-1}\beta([i-1])$ for $1\leq i\leq n$.
  Equation~\eqref{eqn:first} now follows by induction using Pascal's identity $\binom{a}{b}+\binom{a}{b+1}=\binom{a+1}{b+1}$ and the values $\beta([i-1])=\binom{n-1}{i-1}$.

  Assume $\ell\geq 2$ henceforth.
  Assume further that $\alpha=(\alpha_1,\dots,\alpha_{\ell})$ is such that there exists a $k\in [\ell-1]$ with $\alpha_k\geq 2$. Fix such a $k$.
  Define another composition $\alpha'$ of size $n-1$ and length $\ell$ by
  \begin{align}\label{eqn:transform_alpha}
    \alpha'\coloneqq (\alpha_1,\dots, \alpha_{k}-1,\alpha_{k+1}+1,\dots,\alpha_{\ell}).
  \end{align}
  Let $\gamma$ (resp. $\delta$) be the composition obtained by restricted to the first $k$ parts (resp. last $\ell-k$ parts) of $\alpha'$ (resp. $\alpha$).
  We have
  \begin{align}\label{eqn:in_detail}
    M_{\alpha}(\alpx_m)-M_{\alpha'}(\alpx_m)&=\sum_{1\leq i_1<\cdots <i_{\ell}\leq m}x_{i_1}^{\alpha_1}\cdots x_{i_k}^{\alpha_k-1}(x_{i_k}-x_{i_{k+1}})x_{i_{k+1}}^{\alpha_{k+1}}\cdots x_{i_{\ell}}^{\alpha_{\ell}}
    \nonumber
    \\
    &=\sum_{1\leq i_1<\cdots <i_{\ell}\leq m}x_{i_1}^{\gamma_1}\cdots x_{i_k}^{\gamma_k}\left(\sum_{r=i_k}^{i_{k+1}-1}(x_{r}-x_{r+1})\right)x_{i_{k+1}}^{\delta_{1}}\cdots x_{i_{\ell}}^{\delta_{\ell-k}}
    \nonumber
    \\
    &=\sum_{r=1}^{m-1}\sum_{\substack{1\leq i_1<\cdots <i_k\leq r\\r< i_{k+1}<\cdots <i_{\ell}\leq m}}x_{i_1}^{\gamma_1}\cdots x_{i_k}^{\gamma_k}(x_{r}-x_{r+1})x_{i_{k+1}}^{\delta_{1}}\cdots x_{i_{\ell}}^{\delta_{\ell-k}}
    \nonumber
    \\
    &=\sum_{r=1}^{m-1} M_{\gamma}(x_1,\dots,x_r)(x_r-x_{r+1})M_{\delta}(x_{r+1},\dots,x_m).
  \end{align}

Set $n_1\coloneqq \sum_{1\leq i\leq k}\alpha_i=|\gamma|+1$ and $n_2\coloneqq \sum_{k<i\leq \ell }\alpha'_k=|\delta|+1$.
Since $M_\gamma(x_1,\ldots,x_r)$ has degree $|\gamma|=n_1-1$, by Corollary~\ref{cor:fundamental_cor} we obtain
\begin{align}\label{eqn:second}
  \ds{M_{\alpha}(\alpx_m)-M_{\alpha'}(\alpx_m)}{n}=
  \binom{n}{n_1}
  \ds{M_\gamma(x_1,\ldots,x_{n_1})}{n_1}\ds{M_\delta(x_{1},\ldots,x_{m-n_1})}{n_2}.
\end{align}
  Now $\ell(\gamma)=k<\ell$, so by induction $\ds{M_\gamma(x_1,\ldots,x_{n_1})}{n_1}=0$ using \eqref{eqn:divided symmetrization monomial quasisymmetric 0}.
Therefore the expression in ~\eqref{eqn:second} vanishes, and
we conclude that $\ds{M_{\alpha}(x_1,\ldots,x_m)}{n}=\ds{M_{\beta}(\alpx_m)}{n}$, whenever $\alpha$ and $\alpha'$ are related via the transformation specified in \eqref{eqn:transform_alpha}.
By performing a series of such transformations, one can transform any composition $\alpha$ of length $\ell$ into the hook composition $(1^{\ell-1},n-\ell)$.

It follows that {\em $\ds{M_{\alpha}(\alpx_m)}{n}$ depends only on $\ell$, $m$, and $n$.}
To conclude, we distinguish the cases $m=n$ and $m<n$.
\smallskip

\noindent \textbf{Case I: }
    If $m=n$, then for a fixed $\ell\geq 2$ consider
    $\displaystyle{
    M_{\ell}\coloneqq \sum_{\alpha\vDash n-1,~\ell(\alpha)=\ell}M_\alpha(\alpx_n).}$

    Note that $M_\ell\in \Lambda_n$ since it is the sum of all monomials $\alpx^{\bf c}$ where ${\bf c}$ is a weak composition of $n-1$ with $\ell$ nonzero integers. Thus $\ds{M_{\ell}}{n}=0$ by Corollary~\ref{cor:symmetric_factor}. It follows that
    $\ds{M_\alpha(\alpx_n)}{n}=0$ for any composition $\alpha$ with length $\ell$, as desired.
\smallskip

\noindent  \textbf{Case II: }  If $m<n$, we compute the common value using the special composition $\alpha=(n-\ell,1^{\ell-1})$. By its definition, $M_{\alpha}(\alpx_m)$ is the sum of all monomials of the form $x_{i}^{n-\ell} x_{j_1}\cdots x_{j_{\ell-1}}$ where $1\leq i\leq m-\ell+1$ and $i<j_1<\cdots <j_{\ell-1}\leq m$. 
We now want to apply Lemma~\ref{lem:ds_monomial} to each such monomial $\alpx^{\bf c}$, which requires to compute the sets $S_{\bf c}$. 

    Now $\alpha$ was picked so that we have the simple equality  $S_{\bf c}=[i-1]$.  Indeed $\mathrm{psum}_j({\bf c})=-j<0$ for $j=1,\ldots, i-1$ so that $[i-1]\subseteq S_{\bf c}$. For $j=i,\ldots, n-1$, we have $\mathrm{psum}_j({\bf c})\geq 0$: this is perhaps most easily seen on the path $P({\bf c})$, since starting at abscissa $i$, only level and down steps occur, with the last step being a down step to the point $(n,-1)$. This last fact follows from $m<n$, which is the reason why we needed to treat the case $m=n$ separately in the proof.

    We know that $\beta_n([i-1])=\binom{n-1}{i-1}$, and since there are  $\binom{m-i}{\ell-1}$ choices for $\{j_1,\cdots,j_{\ell-1}\}$ for a given $i$, we obtain
\begin{align*}
  \ds{M_\alpha(\alpx_m)}{n}=\sum_{i=1}^{m-\ell+1}(-1)^{i-1}\binom{n-1}{i-1}\binom{m-i}{\ell-1}.
\end{align*}

To finish the proof, it remains to show the identity
\begin{align}\label{eqn:third}
  \sum_{i=1}^{m-\ell+1}(-1)^{i-1}\binom{n-1}{i-1}\binom{m-i}{\ell-1}=(-1)^{m-\ell}\binom{n-1-\ell}{m-\ell}.
\end{align}

This is done in Appendix~\ref{app:identity}.
\end{proof}
In the next subsection, we will see how Proposition~\ref{prop:DS_Malpha_general} implies a pleasant result (Theorem~\ref{thm:intro_main_1}) for all quasisymmetric polynomials.

\subsection{Divided symmetrization of quasisymmetric polynomials}
\label{subsec:all_the_action}
In this subsection, we give a natural interpretation of $\ds{f(x_1,\ldots,x_m)}{n}$ for $m\leq n$ when $f$ is a quasisymmetric polynomial in $x_1,\dots,x_m$ with degree $n-1$.
To this end, we briefly discuss a generalization of Eulerian numbers that is pertinent for us.

If $\phi(x)\in \bQ[x]$  satisfies $\deg(\phi)<n$, then (cf. \cite[Chapter 4]{St97}) there exist $h^{(n)}_m(\phi)\in \bQ$ such that
\begin{align}\label{eq:h-eulerian-gf}
\sum_{j\geq 1}\phi(j) t^j=\frac{\sum_{m=0}^{n-1}h^{(n)}_m(\phi) t^m}{(1-t)^n}.
\end{align}
By extracting coefficients, the $h^{(n)}_m(\phi)$ are uniquely determined by the following formulas for $j=0,\ldots,n-1$:
\begin{align}\label{eq:h-eulerian-formula}
\phi(j)=\sum_{i=0}^j\binom{n-1+i}{i}h^{(n)}_{j-i}(\phi).
\end{align}
Stanley calls the $h^{(n)}_i(\phi)$ the \emph{$\phi$-eulerian numbers} (cf. \cite[Chapter 4.3]{St97}), and the numerator the \emph{$\phi$-eulerian polynomial}. Indeed if $\phi(j)=j^{n-1}$ we get the classical Eulerian numbers $A_{n,i}$ counting permutations with $i-1$ descents, and the Eulerian polynomial $A_n(t)$.\medskip

Let $\Qsym$ denote the ring of quasisymmetric functions.
Let ${\bf x}$ denote the infinite set of variables $\{x_1,x_2,\dots\}$.
Elements of $\Qsym$ may be regarded as bounded-degree formal power series $f\in \bQ[\![{\bf x}]\!]$ such that for any composition $(\alpha_1,\dots,\alpha_k)$ the coefficient of $x_{i_1}^{\alpha_1}\cdots x_{i_k}^{\alpha_k}$ equals that of $x_{1}^{\alpha_1}\cdots x_{k}^{\alpha_k}$ whenever $i_1<\cdots<i_k$. 

Equivalently, $\Qsym$ is obtained by taking the inverse limit of the rings $\{\Qsym_m\}$ in finitely many variables. Given $f\in\Qsym$, we thus have the natural  {\em restriction} to $\Qsym_m$:
  \begin{align*}
    f(x_1,\ldots,x_m)\coloneqq f(x_1,\ldots,x_m,0,0,\ldots).
  \end{align*}
Additionally, set $f(1^m)\coloneqq f(1,\dots,1)$ by specializing the previous polynomial.  The operations $f\mapsto f(x_1,\ldots,x_m)$ and $f\mapsto f(1^m)$ are ring morphisms, which are denoted $r_m(f)$ and $ps_m^1(f)$ respectively in  \cite[Section 7.8]{St99} in the context of symmetric functions.

Denote the degree $n-1$ homogeneous summand of $\Qsym$ by $\Qsym^{(n-1)}$, and let $f\in \Qsym^{(n-1)}$. Observe that \emph{$f(1^m)$ is a polynomial in $m$ of degree at most $n-1$}. By linearity is enough to check this on a basis of $\Qsym^{(n-1)}$. We pick the quasisymmetric monomials $M_\alpha({\bf x})$ (see~\eqref{eq:defi_quasi_monomials}), so that $M_\alpha(1^m)$ is the number of monomials in $M_\alpha(\alpx_m)$, that is
\begin{align}
\label{eq:poly_Malpha}
M_{\alpha}(1^m)=\binom{m}{\ell(\alpha)},
\end{align}
 evidently a polynomial in $m$ of degree $\ell(\alpha)\leq n-1$. 

Given $f\in \Qsym^{(n-1)}$, let $\phi_f$ denote the polynomial of degree $<n$ such that $\phi_f(m)=f(1^m)$ for $m\geq 1$. The $\phi_f$-Eulerian numbers $h^{(n)}_m(\phi_f)$ are thus well defined for $1\leq m\leq n-1$.  Our first main result is that these can be obtained by divided symmetrization:

{
\renewcommand{\thetheorem}{\ref{thm:intro_main_1}}
\begin{theorem}
If $f\in \Qsym^{(n-1)}$, then   we have the identity
\begin{align}\label{eqn:generic_statement_quasisymmetric}
\sum_{j\geq 1}f(1^j) t^j=\frac{\sum_{m=1}^{n}\ds{f(x_1,\ldots,x_m)}{n}\;t^m}{(1-t)^n}.
\end{align}
Equivalently, we have $\ds{f(x_1,\ldots,x_m)}{n}=h^{(n)}_m(\phi_f)$ for $1\leq m\leq n-1$, and $\ds{f(x_1,\ldots,x_n)}{n}=0$.
\end{theorem}
\addtocounter{theorem}{-1}
}

\begin{proof}
By linearity, it suffices to prove this for any basis of $\Qsym^{(n-1)}$, and we pick the monomial basis. We have  $\displaystyle{{M_{\alpha}}(1^j)=\binom{j}{\ell(\alpha)}}$ by \eqref{eq:poly_Malpha}, so by summing we get
\begin{align}
  \sum_{j\geq 1}\binom{j}{\ell(\alpha)}t^j=\frac{t^{\ell(\alpha)}}{(1-t)^{\ell(\alpha)+1}}=\frac{t^{\ell(\alpha)}(1-t)^{n-\ell(\alpha)-1}}{(1-t)^n}
\end{align}
The coefficient of $t^m$ in the numerator is given by $(-1)^{m-\ell(\alpha)}\binom{n-1-\ell(\alpha)}{m-\ell(\alpha)}$ for $1\leq m\leq n-1$, and $0$ for $m=n$. These are precisely the values of $\ds{{M}_{\alpha}(x_1,\ldots,x_m)}{n}$ computed in Proposition~\ref{prop:DS_Malpha_general}, thereby completing the proof.
\end{proof}

\begin{remark}
Using the identities~\eqref{eq:h-eulerian-formula}, one obtains for $j\geq 1$:
\begin{align}\label{eq:blabla}
f(1^j)=\sum_{i=1}^j\binom{n-1+i}{i}\ds{f(x_1,\ldots,x_i)}{n},
\end{align}
which is easily inverted to give for $m\leq n$:
\begin{align}\label{eq:blablabis}
\ds{f(x_1,\ldots,x_m)}{n}=\sum_{i=0}^{m-1}(-1)^{i}\binom{n}{i}f(1^{m-i}).
\end{align}
Note that both sides only depend on the quasisymmetric polynomial $f(x_1,\ldots,x_m)$, and not the whole quasisymmetric function. Since any quasisymmetric polynomial in $\bQ[\alpxn{m}]$ can be written as such a restriction $f(x_1,\ldots,x_m)$, it follows that \eqref{eq:blablabis} gives a way to compute the divided symmetrization of any quasisymmetric polynomial in $\bQ[\alpxn{m}]$.
\end{remark}

We can apply Theorem~\ref{thm:intro_main_1} to some special families of symmetric and quasisymmetric functions.
\begin{example}
\label{ex:symmetric_p_and_m}
Fix $\lambda\vdash n-1$ and set $\ell\coloneqq \ell(\lambda)$.
Consider the power sum symmetric function $p_{\lambda}$.
For $j\geq 0$, we have that $\phi_{p_{\lambda}}(j)=j^{\ell}$ and therefore:
  \[
  \sum_{j\geq 0}\phi_{p_\lambda}(j) t^j=\frac{A_\ell(t)}{(1-t)^{\ell+1}}
  =\frac{A_\ell(t)(1-t)^{n-\ell-1}}{(1-t)^{n}}
  \]
  Theorem~\ref{thm:intro_main_1} implies that
  \begin{align}\label{eqn:divided symmetrization power symmetric}
  \ds{p_\lambda(x_1,\ldots,x_m)}{n}=\sum_{i=1}^{\min(\ell(\lambda),m)} A_{\ell(\lambda),i}(-1)^{m-i}\binom{n-1-\ell(\lambda)}{m-i}.
  \end{align}
  
  \end{example}

A natural question is to study the action of divided symmetrization on other classical bases of symmetric functions. For monomial symmetric functions $m_\lambda$, this is obtained immediately from Proposition~\ref{prop:DS_Malpha_general}. The case of Schur functions $s_\lambda$ is treated in Example~\ref{ex:schur}. As for elementary symmetric functions $e_{\lambda}$ and the homogeneous symmetric functions $h_{\lambda}$, we omit the computations given the unwieldy expressions that one obtains. Instead, we discuss the case of the important basis of $\Qsym$ given by fundamental quasisymmetric functions $F_\alpha$, in which case divided symmetrization behaves as well as one could hope.


%
%

\subsection{Expansions into fundamental quasisymmetric functions.}
\label{subsec:applications}

The following pleasant fact falls out of Theorem~\ref{thm:intro_main_1}.
\begin{proposition}
  \label{prop:F_positive}
  For any $\gamma\vDash n-1$,  and for $m\leq n$, we have
  \begin{align}
  \label{eqn:divided symmetrization fundamental quasisymmetric}
  \ds{F_\gamma(\alpx_m)}{n}=\delta_{m,\ell(\gamma)}.
  \end{align}
 Thus if $f\in \Qsym^{(n-1)}$ expands as $\displaystyle{f=\sum_{\alpha\vDash n-1}c_{\alpha}F_{\alpha}}$, then for any positive integer $m<n$
 \begin{align}\ds{f(\alpx_m)}{n}=\sum_{\substack{\alpha\vDash n-1\\ \ell(\alpha)=m}}c_{\alpha}.
 \end{align}
\end{proposition}

\begin{proof}
  By expanding the definition~\eqref{eq:Falpha}, we have that $F_\gamma(\alpx_j)$ is the sum of all monomials $x_{i_1}\ldots x_{i_{n-1}}$ with $1\leq i_1\leq i_2\leq \cdots\leq i_{n-1}\leq j$ and $i_{t}<i_{t+1}$ whenever $t\in\{\gamma_1,\gamma_1+\gamma_2,\ldots,\gamma_1+\cdots+\gamma_{\ell(\alpha)-1}\}$. By counting these monomials we obtain $F_\gamma(1^j)=\binom{j-\ell(\gamma)+n-1}{n-1}$, and therefore 
  \begin{equation}\label{eqn:lhs_for_fundamentals}
  	\sum_{j\geq 0}F_{\gamma}(1^j)t^j=\sum_{j\geq 0}\binom{j-\ell(\gamma)+n-1}{n-1}t^j\nonumber=\frac{t^{\ell(\gamma)}}{(1-t)^n}.
  \end{equation}
  A comparison of this last expression with \eqref{eqn:generic_statement_quasisymmetric} implies the claim.
\end{proof}

If a quasisymmetric function expands nonnegatively in terms of fundamental quasisymmetric functions, we call it \emph{$F$-positive}. To emphasize the significance of Proposition~\ref{prop:F_positive}, we discuss some examples of $F$-positive quasisymmetric functions arising naturally in the wild.
Our examples fall into two categories broadly, and we do not aim to be exhaustive.\medskip

\subsubsection{$(P,\omega)$-partitions}
Arguably the most natural context in which $F$-positive quasisymmetric functions arise is that of $(P,\omega)$-partitions.
We describe the setup briefly, referring the reader to \cite[Section 7.19]{St99} for a more detailed exposition.
Let $(P,\leq_P)$ be a finite poset on $n-1$ elements.
A \emph{labeling} of $P$ is a bijection $\omega:P\to [n-1]$.
Given a labeled poset $(P,\omega)$, a \emph{$(P,\omega)$-partition} is a map $\gamma:P\to \bZ_{>0}$ with the following properties:
\begin{itemize}
    \item If $i\leq_P j$ and $\omega(i)<\omega(j)$, then $\gamma(i)\leq \gamma(j)$.
    \item If $i\leq_P j$ and $\omega(i)>\omega(j)$, then $\gamma(i)< \gamma(j)$.
\end{itemize}
Let $\mathcal{A}(P,\omega)$ denote the set of all $(P,\omega)$-partitions.
Furthermore, define the \emph{Jordan-H\"older} set $\mathcal{L}(P,\omega)$ to be the set of all permutations $\pi\in S_{n-1}$ such that the map $w:P\to [n-1]$ defined by $w(\omega^{-1}(\pi_j))=j$ is a linear extension of $P$.
Consider the formal power series
\begin{align}
  K_{P,\omega}=\sum_{\gamma\in \mathcal{A}(P,\omega)}\prod_{i\in P}x_{\gamma(i)}.
\end{align}
The central result in Stanley's theory of $(P,\omega)$-partitions is that
\begin{align}
  K_{P,\omega}=\sum_{\pi\in \mathcal{L}(P,\omega)}F_{\comp(\pi)},
\end{align}
where we abuse notation and denote the composition of $n-1$ corresponding to the descent set of $\pi$ by $\comp(\pi)$.
%

Proposition~\ref{prop:F_positive} specializes to the following result.
\begin{corollary}\label{cor:p-partitions}
  Fix a positive integer $n$. Let $(P,\omega)$ be a labeled poset on $n-1$ elements. The following equality holds for $m\leq n$:
  \[
    \ds{K_{P,\omega}(\alpx_m)}{n}=|\{\pi\in \mathcal{L}(P,\omega)\suchthat \pi \text{ has } m-1 \text{ descents }\}|.
  \]
\end{corollary}
Examples of quasisymmetric functions arising from $(P,\omega)$-partitions, either implicitly or explicitly, abound in combinatorics.
For instance, skew Schur functions \cite{Ges84}, chromatic symmetric functions \cite{St95} and their quasisymmetric refinement \cite{SW16}, the matroid quasisymmetric function of Billera-Jia-Reiner \cite{BJR09} are all examples of $F$-positive quasisymmetric functions that can be understood in terms of $(P,\omega)$-partitions.

\begin{example} \label{ex:schur} The Schur functions $s_\lambda$ are of the form $K_{P,\omega}$ where $P=P_\lambda$ is the poset given by the Young diagram of $\lambda$, and $\omega$ labels rows of the diagram from bottom to top, and from left to right in each row.  Thus Corollary~\ref{cor:p-partitions} gives us that when $\lambda\vdash n-1$, $\ds{s_{\lambda}(\alpx_m)}{n}$ is the number of standard Young tableaux of shape $\lambda$ with $m-1$ descents.
\end{example}

\begin{remark}\label{rem:grassmannian}
Recall from the introduction that we were initially interested in the values $a_w$ given by the divided symmetrization of Schubert polynomials $\schub_w$, where $w\in S_n$ has length $n-1$. When $w$ is a \emph{Grassmannian} permutation of shape $\lambda$ and descent $m$, one has $\schub_w=s_{\lambda}(\alpx_m)$. The previous example then shows that the intersection number $a_w$ is the number of standard Young tableaux of shape $\lambda$ with $m-1$ descents.
\end{remark}

\subsubsection{Edge-labeled posets} We proceed to discuss another class of quasisymmetric functions arising from posets, except now the edges in the Hasse diagram have labels rather than the vertices.
An \emph{edge-labeled} poset $P$ is a finite graded poset with unique maximal element $\hat{1}$ and unique minimal element $\hat{0}$ whose cover relations are labeled by integers.
Assume that the rank of $P$ is $n-1$.
Let $\mathcal{C}(P)$ be the set of maximal chains in $P$.
Given $\rho\in\mathcal{C}(P)$, the edge labels in $\rho$ read from $\hat{0}$ to $\hat{1}$ give the word $w(\rho)$ corresponding to $\rho$.
Define $\des(\rho)$ to be the descent set of $w(\rho)$ and $\comp(\rho)$ to be the composition of $n-1$ corresponding to $\des(\rho)$.

Bergeron and Sottile \cite{BS99} define a quasisymmetric function $F_{P}(\alpx)$ by
\begin{align}
    \label{eqn:BS_Fp}
    F_{P}=\sum_{\rho\in \mathcal{C}(P)}F_{\comp(\rho)}.
\end{align}
The result analogous to Corollary~\ref{cor:p-partitions} in the current context is the following.
\begin{corollary}\label{cor:edge-labeled}
  Fix a positive integer $n$. Let $P$ be an edge-labeled poset of rank $n-1$. The following equality holds for $m\leq n$:
  \[
    \ds{F_P(\alpx_m)}{n}=|\{\rho\in \mathcal{C}(P)\suchthat w(\rho) \text{ has } m-1 \text{ descents }\}|.
  \]
\end{corollary}
The quasisymmetric functions $F_P$ can be considered  a common generalization of Stanley symmetric functions, skew Schur functions, and skew Schubert functions by making the appropriate choice of edge-labeled  poset $P$: in this case one picks intervals in the weak Bruhat order, Young's lattice and Grassmannian Bruhat order respectively. Furthermore, in special cases, $F_P$ also equals Ehrenborg's flag quasisymmetric function \cite{Ehr96}. The reader is referred to \cite{BS02} for further motivation to study $F_P$.

We conclude this section by describing the case of the\emph{ Stanley symmetric function}.
Recall that a permutation $v$ covers a permutation $u$ in the (right) weak order on $S_n$ if there exists a simple transposition $s_i$ such that $us_i=v$ and $\ell(v)=\ell(u)+1$. Label this cover relation by $i$, and given $w\in S_n$, consider the interval $P_w=[e,w]$ in the weak order, where $e$ denote the identity permutation in $S_n$. The words read from the maximal chains of $P$ are reduced words for $w$. The corresponding $F_{P_w}$ is the Stanley symmetric function $F_w$, which has degree $\ell(w)$. We thus get the following result:

\begin{proposition} If $w$ has length $n-1$, then $\ds{F_w(\alpx_m)}{n}$ is the number of reduced words for $w$ with $m$ descents.
\end{proposition}

\section{Decomposition of the divided symmetrization operator.}\label{sec:decomposition}
We proceed to another perspective on divided symmetrization, one which relates it to the study of super-covariant polynomials initiated by Aval-Bergeron \cite{AB03} and Aval-Bergeron-Bergeron \cite{ABB04}.

\subsection{Divided symmetrization and the ideal $\mathcal{J}_n$}\label{subsec:structural1}
Let $\mathcal{I}_n$ be the ideal in $\bQ[\alpx_n]$ generated by symmetric polynomials with positive degree. It is  an immediate consequence of Corollary~\ref{cor:symmetric_factor} that $\ds{f}{n}=0$ for any $f$ in $\degn\cap\;\mathcal{I}_n$. It turns out that the result still holds when replacing \emph{symmetric }by \emph{quasisymmetric}, which is the focus of this section.


 \begin{definition} We denote by $\mathcal{J}_n$ the ideal of $\bQ[\alpx_n]$ generated  by homogeneous quasisymmetric polynomials in $x_1,\dots,x_n$ with positive degree.
 \end{definition} 
\noindent Obviously $\mathcal{I}_n\subset \mathcal{J}_n$. Let us define $K_n\coloneqq \degn\cap \mathcal{J}_n$. Then we have the following result:


\begin{proposition}\label{prop:Jn_vanishing}
If $f\in K_n$, then $\ds{f}{n}=0$.
\end{proposition}

We first prove the following lemma.

\begin{lemma}
\label{lem:DS_quasisymmetric}
  Fix a positive integer $n\geq 2$. Let $p$ be a nonnegative integer satisfying $1\leq p\leq n-1$, $\alpha\vDash p$ and ${\bf c}\in\wc{n}{n-1-p}$.
  Then we have
  \[
    \ds{M_{\alpha}(\alpx_n)\alpx^{\bf c}}{n}=0.
  \]
\end{lemma}

 Note that the case $p=n-1$ is exactly the case $m=n$ of Proposition~\ref{prop:DS_Malpha_general}.

\begin{proof}[Proof of Lemma~\ref{lem:DS_quasisymmetric}]
  Broadly speaking, our strategy is similar to that adopted in the proof of Proposition~\ref{prop:DS_Malpha_general}.
We proceed by induction on $n$.
In the case $n=2$, we have $\alpha=(1)$ and ${\bf c}=(0,0)$.
Then $\ds{M_{(1)}(\alpx_2)}{2}=0$ by Corollary~\ref{cor:symmetric_factor} since $M_{(1)}(\alpx_2)=x_1+x_2$ is symmetric.

Let $n\geq 3$ henceforth, and suppose $\alpha=(\alpha_1,\dots,\alpha_{\ell})$.
If $\ell=1$, then we have that $\ds{M_{\alpha}(\alpx_n)\alpx^{\bf c}}{n}=0$ by Corollary~\ref{cor:symmetric_factor} since $M_{\alpha}(\alpx_n)$ is symmetric in this case.
So let us assume $\ell>1$.
Assume further that there exists a $k\in [\ell-1]$ with $\alpha_k\geq 2$. Fix such a $k$.
Define $\beta\vDash n-1$ and length $\ell$ by
\begin{align}\label{eqn:transform_alpha_bis}
  \beta\coloneqq (\alpha_1,\dots, \alpha_{k}-1,\alpha_{k+1}+1,\dots,\alpha_{\ell}).
\end{align}
Let $\gamma$ (resp. $\delta$) be the composition obtained by restricted to the first $k$ parts (resp. last $\ell-k$ parts) of $\beta$ (resp. $\alpha$), and write ${\bf d}=(c_1,\ldots,c_n)$. By mimicking how we arrived at \eqref{eqn:in_detail}, we obtain
\begin{align}
  (M_{\alpha}(\alpx_n)-M_{\beta}(\alpx_n))\alpx^{\bf c}&=
  \sum_{r=1}^{n-1} M_{\gamma}(\alpx_r)\prod_{i=1}^{r}x_i^{c_i}(x_r-x_{r+1})M_{\delta}(x_{r+1},\dots,x_m)\prod_{i=r+1}^{n}x_{i}^{c_{i}}.
\end{align}
This in turn implies
\begin{align}
  \label{eqn:ds_mon_plus_some}
\ds{(M_{\alpha}(\alpx_n)-M_{\beta}(\alpx_n))\alpx^{\bf c}}{n}=\sum_{r=1}^{n-1}
\binom{n}{r}\ds{M_\gamma(\alpx_{r})\prod_{i=1}^{r}x_i^{c_i}}{r}\ds{M_\delta(\alpx_{n-r})\prod_{i=1}^{n-r}x_i^{c_{r+i}}}{n-r}.
\end{align}

We claim that each term on the right hand side of \eqref{eqn:ds_mon_plus_some} vanishes.
Fix an $r$ satisfying $1\leq r\leq n-1$. If $\deg(M_\gamma(x_1,\ldots,x_{r})\prod_{i=1}^{r}x_i^{c_i})\neq r-1$, then Corollary~\ref{cor:fundamental_cor} implies that our claim is true.
So suppose that $\deg(M_\gamma(x_1,\ldots,x_{r})\prod_{i=1}^{r}x_i^{c_i})=r-1$, which is equivalent to $\deg(M_\delta(x_{1},\ldots,x_{n-r})\prod_{i=1}^{n-r}x_i^{c_{r+i}})=n-r-1$.
Observe that the minimum of $\{r,n-r\}$ is at least $2$.
But then our inductive hypothesis implies that at least one of $\ds{M_\gamma(\alpx_{r})\prod_{i=1}^{r}x_i^{c_i}}{r}$ or
$\ds{M_\delta(\alpx_{n-r})\prod_{i=1}^{n-r}x_i^{c_{r+i}}}{n-r}$ equals $0$.
Thus, we have established that $\ds{(M_{\alpha}(\alpx_n)-M_{\beta}(\alpx_n))\alpx^{\bf c}}{n}=0$ indeed.
As in the proof of Proposition~\ref{prop:DS_Malpha_general}, we conclude that $\ds{M_{\alpha}(\alpx_n)\alpx^{\bf c}}{n}$ depends only on $n$, ${\bf c}$, and $ \ell(\alpha)$.

To conclude, note that the sum over all $\alpha\vDash p$ with $\ell$ parts of $M_{\alpha}(\alpx_n)x_1^{c_1}\cdots x_n^{c_n}$ is a polynomial of degree $n-1$ with a symmetric factor, and thus its divided symmetrization yields $0$ by Corollary~\ref{cor:symmetric_factor}. The claim follows.
\end{proof}

\begin{proof}[Proof of Proposition~\ref{prop:Jn_vanishing}]
Let $f\in K_n$. Since the $M_\alpha(\alpx_n)$ for $|\alpha|\geq 1$ are a linear basis of the space of quasisymmetric polynomials in $x_1,\dots,x_n$ with no constant term, any element $f\in K_n$ possesses a decomposition of the form
  \begin{align}\label{eqn:expansion_element_in_Kn}
    \sum_{\substack{\alpha,{\bf c}\\ |\alpha|+|{\bf c}|=n-1}} b_{\alpha,{\bf c}}\alpx^{\bf c}M_{\alpha}(\alpx_n),
  \end{align}
  where $|\alpha|\geq 1$ and $b_{\alpha,{\bf c}}\in \bQ$. By Lemma~\ref{lem:DS_quasisymmetric}, each summand on the right hand side of \eqref{eqn:expansion_element_in_Kn} yields $0$ upon divided symmetrization, and the claim follows.
\end{proof}

\subsection{Divided symmetrization and super-covariant algebra}\label{subsec:structural2}

 We want to go further and consider how divided symmetrization acts outside of $K_n$. It turns out that previous work of Aval-Bergeron-Bergeron \cite[Theorem 4.1]{ABB04} is directly related to this endeavour. 
 
 Consider the set of weak compositions defined by
  \[
    \mathcal{B}_n=\{(c_1,\dots,c_n)\suchthat \sum_{1\leq j\leq i }c_j<i \text{ for all } 1\leq i\leq n\}.
  \]
    Let $L_n=\mathrm{Vect}\left(\alpx^{\bf c}\suchthat {\bf c}\in \mathcal{B}_n \right)$ be the linear span of the corresponding monomials.
\begin{theorem}[\cite{ABB04}]
\label{thm:ABB}
   We have the vector space decomposition $\bQ[\alpx_n]=L_n\oplus \mathcal{J}_n$.
\end{theorem}

Both $\mathcal{J}_n$ and $L_n$ are homogeneous, so by restriction to degree $n-1$ we have $\degn=(L_n\cap \degn)\oplus K_n$. Now we note that $L_n\cap \degn$ has a basis formed of all anti-Catalan monomials (see Section~\ref{subsec:catalan}) and thus has dimension $\mathrm{Cat}_{n-1}$. For instance if $n=3$ it has basis $\{\alpx^{(0,1,1)},\alpx^{(0,0,2)}\}$ and dimension $2=\mathrm{Cat}_2$.

We prefer working with Catalan monomials since their divided symmetrization is equal to $1$. The involution on $\bQ[\alpx_n]$ given by $x_i\mapsto x_{n+1-i}$ for $1\leq i\leq n$ preserves the degree, and leaves the ideal $\mathcal{J}_n$ globally invariant. Therefore it sends any homogeneous subspace complementary to $\mathcal{J}_n$ to another such subspace. Let $L_n'$ be the image of $L_n$, so that $\bQ[\alpx_n]=L_n'\oplus \mathcal{J}_n$ by Theorem~\ref{thm:ABB}. Define $K_n^{\dagger}\coloneqq L'_n\cap \degn=\mathrm{Vect}\left(\alpx^{\bf c}\suchthat {\bf c}\in \catwc{n} \right)$. 
We get finally the vector space decomposition
\begin{align}
\label{eq:Kn_decomposition}
\degn=K_n^{\dagger}\oplus K_n.
\end{align}

We can now state the structural result that characterizes divided symmetrization with respect to the direct sum \eqref{eq:Kn_decomposition}.

{
\renewcommand{\thetheorem}{\ref{thm:intro_main_2}}
\begin{theorem}
If $f\in \degn$ is written $f=g+h$ with $g\in K_n^{\dagger}$ and $h\in K_n$ according to \eqref{eq:Kn_decomposition}, then 
\[
\ds{f}{n}=g(1,\ldots,1).
\]
\end{theorem}
\addtocounter{theorem}{-1}
}

In words, divided symmetrization on $\degn$ can be obtained by first projecting to the first factor in~\eqref{eq:Kn_decomposition}, and then compute the sum of coefficients of the resulting polynomial.

\begin{proof}
 If $f=g+h$ then $\ds{f}{n}=\ds{g}{n}+\ds{f}{n}=\ds{g}{n}$ by Theorem~\ref{thm:intro_main_2} since we have $h\in K_n$. Now by definition of $K_n^{\dagger}$ all the monomials $\alpx^{\bf c}$ in $g$ are Catalan, and thus satisfy $\ds{\alpx^{\bf c}}{n}=1$. This implies that $\ds{g}{n}=g(1,\ldots,1)$ as wanted.
\end{proof}

\begin{example}\label{ex:show_off_structural}
    We know $\ds{x_1x_3}{3}=-2$ by Lemma~\ref{lem:ds_monomial}. Alternatively, note that
  \[
    x_1x_3=x_1(x_1+x_2+x_3)-(x_1^2+x_1x_2),
  \]
  and that $x_1x_2$ and $x_1^2$ are both Catalan monomials. Therefore, we can use $f=x_1(x_1+x_2+x_3)$ and $g=-(x_1^2+x_1x_2)$ in Theorem~\ref{thm:intro_main_2} to conclude again that $\ds{x_1x_3}{3}=-2$.
\end{example}





As a further demonstration of Theorem~\ref{thm:intro_main_2}, we revisit the divided symmetrization of fundamental quasisymmetric polynomials again.

\subsection{Fundamental quasisymmetric polynomials revisited}\label{subsec:F_again}
Before stating the main result in this subsection, we need two operations for compositions.
Given compositions $\gamma=(\gamma_1,\dots,\gamma_{\ell(\gamma)})$ and $\delta=(\delta_1,\dots,\delta_{\ell(\delta)})$, we define their \emph{concatenation} $\gamma\cdot\delta$ and \emph{near-concatenation} $\gamma\odot\delta$ to be $(\gamma_1,\dots,\gamma_{\ell(\gamma)},\delta_1,\dots,\delta_{\ell(\delta)} )$
and $(\gamma_1,\dots,\gamma_{\ell(\gamma)}+\delta_1,\delta_2,\dots,\delta_{\ell(\delta)})$ respectively.
For instance, we have $(3,2)\cdot (1,2)=(3,2,1,2)$ and $(3,1)\odot(1,1,2)=(3,2,1,2)$.

Given finite alphabets $\alpx_n=\{x_1,\dots,x_n\}$ and $\alpy_m=\{y_1,\dots,y_m\}$, define the formal sum $\alpx_n+\alpy_m$ to be the alphabet $\{x_1,\dots,x_n,y_1,\dots,y_m\}$ where the total order is given by $x_1<\cdots<x_n<y_1<\cdots <y_m$.
Following Malvenuto-Reutenauer \cite{MR95}, we have
\begin{align}\label{eqn:fund_coprod}
  F_{\alpha}(\alpx_n+\alpy_m)=\sum_{\gamma\cdot\delta=\alpha \text{ or } \gamma\odot\delta=\alpha}F_{\gamma}(\alpx_n)F_{\delta}(\alpy_m).
\end{align}

\begin{example}
  Interpreting $\{x_1,x_2,x_3\}$ as the sum of $\{x_1\}$ and $\{x_2,x_3\}$ and expanding $F_{(2,3)}(x_1,x_2,x_3)$ using Equation~\eqref{eqn:fund_coprod} gives
   \[
    F_{23}(\alpx_3)=F_{2}(x_1)F_{3}(x_2,x_3)+F_{1}(x_1)F_{13}(x_2,x_3)+F_{23}(x_2,x_3).
  \]
In this expansion, we have suppressed commas and parentheses in writing our compositions, and used the fact that $F_{\alpha}(\alpy)=0$ for any alphabet $\alpy$ with cardinality strictly less than $\ell(\alpha)$.
\end{example}
As explained in \cite[Section 2]{MR95}, the equality in \eqref{eqn:fund_coprod} relies on the coproduct in the Hopf algebra of quasisymmetric functions.
By utilizing the antipode on this Hopf algebra \cite[Corollary 2.3]{MR95}, one can evaluate quasisymmetric functions at formal differences of alphabets.
See \cite[Section 2.3]{AFNT15} for a succinct exposition on the same.
The analogue of \eqref{eqn:fund_coprod} is
\begin{align}
	\label{eqn:fund_coprod_with_antipode}
	F_{\alpha}(\alpx_n-\alpy_m)=\sum_{\gamma\cdot\delta=\alpha \text{ or } \gamma\odot\delta=\alpha}(-1)^{|\delta|}F_{\gamma}(\alpx_n)F_{\delta^t}(\alpy_m),
\end{align}
where $\delta^t\coloneqq \comp([|\delta|-1]\setminus \set(\delta))$.
For instance, if $\delta=(3,2,1,2)\vDash 8$, then $\set(\delta)\subseteq[7]$ is given by $\{3,5,6\}$.
Thus we obtain $\delta^t=\comp(\{1,2,4,7\})=(1,1,2,3,1)$.

To end this section, we have the following result which precises Proposition~\ref{prop:F_positive}:

\begin{proposition}
  \label{prop:ds_fundamental_abb}
  Let $\alpha\vDash n-1$ and let $m$ be a positive integer satisfying $\ell(\alpha)< m\leq n$.  Then $F_{\alpha}(x_1,\dots,x_m)\in\mathcal{J}_n$.
   
\end{proposition}

\begin{proof}
  From \eqref{eqn:fund_coprod_with_antipode} it follows that
  \begin{align}\label{eqn:difference_of_alphabets}
    F_{\alpha}(\alpxn{m})=\sum_{\gamma\cdot\delta=\alpha \text{ or } \gamma\odot\delta=\alpha}(-1)^{|\delta|}F_{\gamma}(\alpxn{n})F_{\delta^{t}}(x_{m+1},\dots,x_n)
  \end{align}
  Modulo $\mathcal{J}_n$, the only term that survives on the right hand side of \eqref{eqn:difference_of_alphabets} corresponds to $\beta=\varnothing$.
  This in turn forces $\gamma=\alpha$.
  Thus we have that $F_{\alpha}(\alpxn{m})$ is equal to $(-1)^{n-1}F_{\alpha^{t}}(x_{m+1},\dots,x_n)$ modulo $\mathcal{J}_n$.

  Now suppose that $m>\ell(\alpha)$. As $\ell(\alpha^{t})=n-\ell(\alpha)<n-m$, we conclude that $F_{\alpha^{t}}(x_{m+1},\dots,x_n)=0$.
  It follows that $ F_{\alpha}(\alpxn{m})\in \mathcal{J}_n$ in this case, and Proposition~\ref{prop:Jn_vanishing} implies that $\ds{F_{\alpha}(\alpxn{m})}{n}=0$.
\end{proof}

We get back the result of Proposition~\ref{prop:F_positive}, namely
\[
    \ds{F_{\alpha}(x_1,\dots,x_m)}{n}=\delta_{m,\ell(\alpha)}.
  \]
 for $m\leq \ell(\alpha)$. Indeed, Proposition~\ref{prop:ds_fundamental_abb} together with Proposition~\ref{prop:Jn_vanishing} implies that $\ds{F_{\alpha}(\alpxn{m})}{n}=0$ for $\ell(\alpha)<m$. On the other hand, if $\ell(\alpha)=m$, then $F_{\alpha}(\alpx_m)=x_1^{\alpha_1}\cdots x_{m}^{\alpha_m}$ is a Catalan monomial, and thus we have $\ds{F_{\alpha}(\alpx_m)}{n} =1$ in this case.

\section*{Acknowledgements}

The authors wish to thank all participants of the seminar on Hessenberg varieties organized by Sara Billey and Alex Woo at the University of Washington in Winter and Spring 2018, from which this work grew.

\appendix

\section{Proof of Lemma~\ref{lem:ds_monomial}}
\label{app:proof}
We want to prove $\ds{\alpx^{{\bf c}}}{n}=(-1)^{|S_{\bf c}|}\beta(S_{\bf c})$ for any  ${\bf c}\in\wcp{n}$. Our proof proceed in two steps.
  \begin{itemize}
    \item First, via a sequence of \emph{moves}, we transform any such ${\bf c}$ into a weak composition ${\bf c'}=(c_1',\dots,c_n')\in \wcp{n}$ with the properties that $S_{\bf c}=S_{{\bf c'}}$ and $\sum_{1\leq j\leq i}(c'_i-1) \in \{0,-1\}$ for all $1\leq i\leq n$.
    Furthermore, our moves ensure that  $\ds{\alpx^{{\bf c}}}{n}=\ds{\alpx^{{\bf c'}}}{n}$.
    \item Second, we compute $\ds{\alpx^{{\bf c'}}}{n}$ explicitly by exploiting
    a relation satisfied by the numbers $(-1)^{|S|}\beta(S)$.
  \end{itemize}
  
We now furnish details. Let ${\bf c}\in\wcp{n}$, and assume that there exists an index $i\in [n]$ such that $\mathrm{psum}_i({\bf c})\notin \{0,-1\}$.
Let $k$ be the largest such index.
Note that we must have $k\leq n-1$ as ${\bf c}\in\wcp{n} $.
Consider the move sending ${\bf c}$ to a  sequence ${\bf d}$ as follows:
\begin{align}
  (c_1,\dots,c_k,c_{k+1},\dots,c_n) \mapsto
  \left\lbrace
  \begin{array}{ll}
    (c_1,\dots,c_k{-}1,c_{k+1}{+}1,\dots, c_n) & \text{ if } \mathrm{psum}_k({\bf c})>0,\\
    (c_1,\dots,c_k{+}1,c_{k+1}{-}1,\dots, c_n) & \text{ if } \mathrm{psum}_k({\bf c})<-1.\\
  \end{array}\right.
\end{align}
In the case $\mathrm{psum}_k({\bf c})>0$, the sequence ${\bf d}$ is clearly a weak composition of size $n-1$. If $\mathrm{psum}_k({\bf c})<-1$, then the maximality assumption on $k$ along with the fact that $\mathrm{psum}_n({\bf c})=-1$ implies that $c_{k+1}\geq 1$.
Thus, the sequence ${\bf d}$ is a weak composition of $n-1$ in this case as well.
It is easy to see that $S_{\bf c}=S_{\bf d}$.
We show that $\ds{\alpx^{{\bf c}}}{n}=\ds{\alpx^{{\bf d}}}{n}$.
Assume that $\mathrm{psum}_k({\bf c})>0$ and thus ${\bf d}=(c_1,\dots,c_k-1,c_{k+1}+1,\dots, c_n)$.
We have
\begin{align}
  \label{eqn:equality_under_moves}
  \ds{\alpx^{{\bf c}}-\alpx^{{\bf d}}}{n}=\ds{x_1^{c_1}\cdots x_{k}^{c_k-1}(x_k-x_{k+1})x_{k+1}^{c_{k+1}}\cdots x_n^{c_n}}{n}.
\end{align}
By our hypothesis that $\mathrm{psum}_k({\bf c})>0$ , we know that $\deg(x_1^{c_1}\cdots x_{k}^{c_k-1})\geq k$.
Corollary~\ref{cor:fundamental_cor} implies that the right hand side of \eqref{eqn:equality_under_moves} equals $0$, which in turn implies that
$\ds{\alpx^{{\bf c}}}{n}=\ds{\alpx^{{\bf d}}}{n}$.
The case when $\mathrm{psum}_k({\bf c})<-1$ is handled similarly, and we leave the details to the interested reader.

By applying the aforementioned moves repeatedly, one can transform ${\bf c}$ into ${\bf c'}$ with the property that $\mathrm{psum}_i({\bf c'})\in \{0,-1\}$ for all $i\in [n]$.
We are additionally guaranteed that $S_{\bf c}=S_{\bf c'}$ and $\ds{\alpx^{{\bf c}}}{n}=\ds{\alpx^{{\bf c'}}}{n}$.
Figure~\ref{fig:path_with_same_Sc} shows the path corresponding to ${\bf c'}$ where ${\bf c}=(0,3,0,0,0,1,3,0)$ is the weak composition from Figure~\ref{fig:path_to_compute_Sc}.
Explicitly, the moves in going from ${\bf c}$ to ${\bf c'}$ are $(0,3,0,0,0,1,3,0)\to (0,3,0,0,0,2,2,0)\to (0,3,0,0,1,1,2,0)\to (0,2,1,0,1,1,2,0)$.

\begin{figure}[ht]
  \begin{tikzpicture}[scale=.6]
    \draw[gray,very thin] (0,0) grid (10,4);
    \draw[line width=0.25mm, black, <->] (0,2)--(10,2);
    \draw[line width=0.25mm, black, <->] (1,0)--(1,4);
    \node[draw, circle,minimum size=5pt,inner sep=0pt, outer sep=0pt, fill=blue] at (1, 2)   (a) {};
    \node[draw, circle,minimum size=5pt,inner sep=0pt, outer sep=0pt, fill=red] at (2, 1)   (c) {};
    \node[draw, circle,minimum size=5pt,inner sep=0pt, outer sep=0pt, fill=blue] at (3, 2)   (d) {};
    \node[draw, circle,minimum size=5pt,inner sep=0pt, outer sep=0pt, fill=blue] at (4, 2)   (e) {};
    \node[draw, circle,minimum size=5pt,inner sep=0pt, outer sep=0pt, fill=red] at (5, 1)   (f) {};
    \node[draw, circle,minimum size=5pt,inner sep=0pt, outer sep=0pt, fill=red] at (6, 1)   (g) {};
    \node[draw, circle,minimum size=5pt,inner sep=0pt, outer sep=0pt, fill=red] at (7, 1)   (h) {};
    \node[draw, circle,minimum size=5pt,inner sep=0pt, outer sep=0pt, fill=blue] at (8, 2)   (i) {};
    \node[draw, circle,minimum size=5pt,inner sep=0pt, outer sep=0pt, fill=blue] at (9, 1)   (j) {};
    \draw[blue, line width=0.7mm] (a)--(c);
    \draw[blue, line width=0.7mm] (c)--(d);
    \draw[blue, line width=0.7mm] (d)--(e);
    \draw[blue, line width=0.7mm] (e)--(f);
    \draw[blue, line width=0.7mm] (f)--(g);
    \draw[blue, line width=0.7mm] (g)--(h);
    \draw[blue, line width=0.7mm] (h)--(i);
    \draw[blue, line width=0.7mm] (i)--(j);
  \end{tikzpicture}
  \caption{The path corresponding to ${\bf c}=(0,2,1,0,1,1,2,0)$ with $S_{\bf c}=\{1,4,5,6\}$.}
  \label{fig:path_with_same_Sc}
\end{figure}
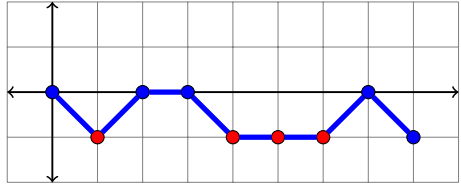

The upshot of the preceding discussion is that to compute $\ds{\alpx^{\bf c}}{n}$ it suffices to consider ${\bf c}\in \wcp{n}$ such that  $\mathrm{psum}_{i}({\bf c})\in \{0,-1\}$ for all $i\in [n]$. The map ${\bf c}\mapsto S_{\bf c}$ restricted to these sequences
is a 1--1 correspondence with subsets $S\subseteq [n-1]$. We write $S\mapsto c(S)$ for the inverse map, and set $\alpx(S)\coloneqq \alpx^{c(S)}$.

Our goal now is to prove that
\begin{align}
  \label{eqn:monomial_indexed_by_S}
  \ds{\alpx(S)}{n}=(-1)^{|S|}\beta(S)
\end{align}
To this end, we proceed by induction on $n$.
When $n=1$, we have $S=\emptyset$. In this case, we have $\alpx(S)=1$, and both sides of the equality in \eqref{eqn:monomial_indexed_by_S} equal $1$.
Let $n\geq 2$ henceforth.
 Assume further that $S\neq [n-1]$, and let 
 $i\notin S$. Define $S_i=S\cap [i]$ and $S^i=\{j\in [n-i]\suchthat j+i\in S\}$. Then Corollary~\ref{cor:fundamental_cor} gives
\begin{align}
\label{eq:recurrence xS}
\ds{\alpx(S)}{n}-\ds{\alpx(S\cup \{i\})}{n}=\binom{n}{i} \ds{\alpx(S_i)}{i} 
\ds{\alpx(S^i)}{n-i}.
\end{align}
The numbers $(-1)^{|S|}\beta(S)$ also satisfy this identity, because:
\begin{align}
  \beta_n(S)+\beta_n(S\cup\{i\})=\binom{n}{i} \beta_i(S_i) 
  \beta_{n-i}\left(S^i\right).
\end{align}
This has a simple combinatorial proof: given a permutation corresponding to the left hand side, split its $1$-line notation after position $i$, and standardize both halves so that they become permutations on $[1,i]$ and $[1,n-i]$ respectively.

To conclude, note that \eqref{eq:recurrence xS} determines all values $\ds{\alpx(S)}{n}$ by induction in terms of the single value $\ds{\alpx([n-1])}{n}$. Now $\alpx([n-1])=x_2\cdots x_n$ and we have $\ds{x_2\cdots x_n}{n}=(-1)^{n-1}$ by Example~\ref{ex:basic_computation}. Since $\beta([n-1])=1$, we have $\ds{\alpx([n-1])}{n}=(-1)^{|[n-1]|}\beta([n-1])$, which completes the proof.

\section{Proof of Identity~\eqref{eqn:third}}
\label{app:identity}

We want to prove that for any positive integers $\ell,m,n$ there holds
\[\sum_{i=1}^{m-\ell+1}(-1)^{i-1}\binom{n-1}{i-1}\binom{m-i}{\ell-1}=(-1)^{m-\ell}\binom{n-1-\ell}{m-\ell}.
\]

Let $P_{\ell,m}$ be the expression on  the left hand side. We prove that it equals the right hand side by induction on $\ell+m$. The base case is $m=\ell=1$, and $P_{1,1}= 1$ as wanted. Assume by induction that the property is valid for all $\ell,m$ such that $\ell+m<k$ for a certain $k\geq 1$, Let $\ell,m'=m+1$ be such that $\ell+m+1=k$. Then we have the following sequence of equalities:
\begin{align*}
P_{\ell,m+1}&=\sum_{i=0}^{m+1-\ell}(-1)^i\binom{n-1}{i}\binom{m-i}{\ell-1}\nonumber\\
& =(-1)^{m-\ell+1}\binom{n-1}{m+1-\ell}+\sum_{i=0}^{m-\ell}(-1)^i\binom{n-1}{i}\left(\binom{m-i}{\ell-1}-\binom{m-i-1}{\ell-1}+\binom{m-i-1}{\ell-1}\right)\nonumber\\
&=(-1)^{m-\ell+1}\binom{n-1}{m+1-\ell}+P_{\ell,m}+P_{\ell-1,m}-(-1)^{m-\ell+1}\binom{n-1}{m+1-\ell}\nonumber\\
&=(-1)^{m-\ell+1}\left(\binom{n-1-\ell}{m+1-\ell}-\binom{n-\ell}{m-\ell}\right),
\end{align*}
where we used the induction hypothesis in the last equality.

 The final expression equals $ (-1)^{m-\ell+1}\binom{n-1-\ell}{m+1-\ell}$ by Pascal's identity, thereby finishing the proof.

\bibliographystyle{acm}
\bibliography{Biblio_DS}

\begin{thebibliography}{10}

\bibitem{Amd16}
{\sc Amdeberhan, T.}
\newblock Explicit computations with the divided symmetrization operator.
\newblock {\em Proc. Amer. Math. Soc. 144}, 7 (2016), 2799--2810.

\bibitem{ABB04}
{\sc Aval, J.-C., Bergeron, F., and Bergeron, N.}
\newblock Ideals of quasi-symmetric functions and super-covariant polynomials
  for {$S_n$}.
\newblock {\em Adv. Math. 181}, 2 (2004), 353--367.

\bibitem{AB03}
{\sc Aval, J.-C., and Bergeron, N.}
\newblock Catalan paths and quasi-symmetric functions.
\newblock {\em Proc. Amer. Math. Soc. 131}, 4 (2003), 1053--1062.

\bibitem{AFNT15}
{\sc Aval, J.~C., F\'{e}ray, V., Novelli, J.~C., and Thibon, J.~Y.}
\newblock Quasi-symmetric functions as polynomial functions on {Y}oung
  diagrams.
\newblock {\em J. Algebraic Combin. 41}, 3 (2015), 669--706.

\bibitem{BS99}
{\sc Bergeron, N., and Sottile, F.}
\newblock Hopf algebras and edge-labeled posets.
\newblock {\em J. Algebra 216}, 2 (1999), 641--651.

\bibitem{BS02}
{\sc Bergeron, N., and Sottile, F.}
\newblock Skew {S}chubert functions and the {P}ieri formula for flag manifolds.
\newblock {\em Trans. Amer. Math. Soc. 354}, 2 (2002), 651--673.

\bibitem{BJR09}
{\sc Billera, L.~J., Jia, N., and Reiner, V.}
\newblock A quasisymmetric function for matroids.
\newblock {\em European J. Combin. 30}, 8 (2009), 1727--1757.

\bibitem{Ehr96}
{\sc Ehrenborg, R.}
\newblock On posets and {H}opf algebras.
\newblock {\em Adv. Math. 119}, 1 (1996), 1--25.

\bibitem{Ges84}
{\sc Gessel, I.~M.}
\newblock Multipartite {$P$}-partitions and inner products of skew {S}chur
  functions.
\newblock In {\em Combinatorics and algebra ({B}oulder, {C}olo., 1983)},
  vol.~34 of {\em Contemp. Math.} Amer. Math. Soc., Providence, RI, 1984,
  pp.~289--317.

\bibitem{Mac95}
{\sc Macdonald, I.~G.}
\newblock {\em Symmetric functions and {H}all polynomials}, second~ed.
\newblock Oxford Mathematical Monographs. The Clarendon Press, Oxford
  University Press, New York, 1995.
\newblock With contributions by A. Zelevinsky, Oxford Science Publications.

\bibitem{MR95}
{\sc Malvenuto, C., and Reutenauer, C.}
\newblock Duality between quasi-symmetric functions and the {S}olomon descent
  algebra.
\newblock {\em J. Algebra 177}, 3 (1995), 967--982.

\bibitem{Pet18}
{\sc Petrov, F.}
\newblock Combinatorial and probabilistic formulae for divided symmetrization.
\newblock {\em Discrete Math. 341}, 2 (2018), 336--340.

\bibitem{Pos09}
{\sc Postnikov, A.}
\newblock Permutohedra, associahedra, and beyond.
\newblock {\em Int. Math. Res. Not. IMRN}, 6 (2009), 1026--1106.

\bibitem{SW16}
{\sc Shareshian, J., and Wachs, M.~L.}
\newblock Chromatic quasisymmetric functions.
\newblock {\em Adv. Math. 295\/} (2016), 497--551.

\bibitem{St95}
{\sc Stanley, R.~P.}
\newblock A symmetric function generalization of the chromatic polynomial of a
  graph.
\newblock {\em Adv. Math. 111}, 1 (1995), 166--194.

\bibitem{St97}
{\sc Stanley, R.~P.}
\newblock {\em Enumerative combinatorics. {V}ol. 1}, vol.~49 of {\em Cambridge
  Studies in Advanced Mathematics}.
\newblock Cambridge University Press, Cambridge, 1997.
\newblock With a foreword by Gian-Carlo Rota, Corrected reprint of the 1986
  original.

\bibitem{St99}
{\sc Stanley, R.~P.}
\newblock {\em Enumerative combinatorics. {V}ol. 2}, vol.~62 of {\em Cambridge
  Studies in Advanced Mathematics}.
\newblock Cambridge University Press, Cambridge, 1999.
\newblock With a foreword by Gian-Carlo Rota and appendix 1 by Sergey Fomin.

\end{thebibliography}

\end{document}